\documentclass[a4paper,english,reqno]{amsart}

\usepackage{amssymb,amsthm, amsmath,amsfonts}
\usepackage{bbm}
\usepackage{graphicx}
\graphicspath{{figs/}}
\usepackage{subfigure}
\usepackage{microtype}
\usepackage{booktabs}
\usepackage{enumerate}
\usepackage{pgf}
\usepackage[latin1]{inputenc}
\usepackage{verbatim}
\usepackage{cite}
\usepackage{url}
\usepackage{xcolor}

\usepackage[cmtip,all]{xy}
\xyoption{arc}

\usepackage{pgf,tikz}
\usetikzlibrary{arrows,arrows.meta,shapes,calc,positioning,decorations.markings}
\tikzset{midarrow/.style={postaction={decorate},
    decoration={markings,mark=at position #1 with {\arrow{Stealth[length=2.6mm]}}}}}
\usepackage{manfnt}
\usepackage{tikz-cd}
\usepackage{url}
\usepackage{microtype}

\usepackage[font=footnotesize]{caption}

\usepackage{xcolor}
\usepackage{hyperref}

\DeclareMathOperator{\reals}{\mathbb{R}}

\newenvironment{myproof}[2] {\emph{Proof of {#1} {#2}.}}{\hfill$\square$}

\newcommand{\R}{\mathbb{R}}

\renewcommand{\H}{\mathbb{H}}

\newtheorem{theorem}{Theorem}
\newtheorem{lemma}{Lemma}
\newtheorem{proposition}{Proposition}
\newtheorem{corollary}{Corollary}

\theoremstyle{definition}
\newtheorem{definition}[theorem]{Definition}
\newtheorem{example}[theorem]{Example}

\newtheorem*{theoremA*}{Main Theorem}
\newtheorem*{theoremB*}{Main Theorem for a Flute}

\numberwithin{equation}{section}

\begin{document}
	
	\title{Strong stable sets and asymptotic fineness}
	
	\author[ S. Burniol]{ Sergi Burniol Clotet}	
	\address{IMERL, CMAT, Universidad de la Rep\'ublica, IFUMI, IRL CNRS 2030, Uruguay}
	
	\email{sergi.burniol@gmail.com}
	
		\author[F. Dal'Bo]{Fran\c{c}oise Dal'Bo}	
	\address{IRMAR, Universit\'e de Rennes, France and IFUMI, IRL CNRS 2030
		}
	\email{francoise.dalbo@univ-rennes.fr}
	
	\subjclass[2020]{Primary 37D40; Secondary 37D20.}
	
	\date{27 October 2026}

	\dedicatory{}

	\begin{abstract}
		We investigate the strong stable sets of the geodesic flow on hyperbolic surfaces and their relation to the injectivity radius. We prove that the strong stable set of a vector does not coincide with its stable horocycle if and only if the injectivity radius along the associated geodesic ray is not bounded away from zero and the ray does not diverge through a cusp. In that case, the strong stable set is an uncountable union of stable horocycles. In previous works, we considered the particular cases where the ray meets arbitrarily short closed geodesics or closed horocycles. Here we complete the study by treating the third possible case, where the ray meets closed hypercycles of arbitrarily small length.
	\end{abstract}
	
	\maketitle

	\section{Introduction}

	We consider the geodesic flow $g_t$ on the unit tangent bundle $T^1S$ of a hyperbolic surface $S$. We endow $T^1S$ with the Sasaki distance $d_{Sa}$ induced by the hyperbolic metric of $S$. For $u\in T^1 S$, the stable horocycle $H^{ss}u$ is defined as the projection of the stable horocycle of any lift of $u$ to $T^1 \H$, whereas the strong stable set $W^{ss} u$ is the set
	$$
	W^{ss}u=\{ v\in T^1S:~ \lim_{t\to +\infty} d_{Sa}(g_tu,g_tv)=0 \}.
	$$

	Following Bellis \cite{Bellis2018}, we define the \emph{asymptotic fineness} of the geodesic ray $u(\R^+)$ as
	\[
	\liminf_{t\ge 0}\mathrm{Inj}\bigl(u(t)\bigr),
	\]
	where $\mathrm{Inj}(x)$ denotes the injectivity radius of $S$ at $x\in S$. 
	Our main result establishes a link between the asymptotic fineness and strong stable sets.

	\begin{theorem}\label{thm1}
		Let $S$ be a hyperbolic surface and $u\in T^1S$. Then $H^{ss}u\subsetneq W^{ss}u$ if and only if
		\[
		\liminf_{t\ge 0}\mathrm{Inj}\bigl(u(t)\bigr)=0
		\]
		and the geodesic ray $u(\R^+)$ does not diverge through a cusp. In that case, the set $W^{ss}u$ is an uncountable union of stable horocycles $H^{ss} v_i$.
	\end{theorem}

	There are several ways for the asymptotic fineness of $u(\R^+)$ to vanish. Bellis \cite{BellisThesis2018} and the authors together with S. Herrero \cite{BurniolDalboHerrero2026} study the case where $u(\R^+)$ meets a sequence of arbitrarily short closed geodesics. In \cite{BurniolDalbo2026}, the authors study the case where $u(\R^+)$ meets a sequence of arbitrarily short closed horocycles. The contribution of this article is to study a third possible case, which is not always covered by the previous works: the ray $u(\R^+)$ meets closed hypercycles of arbitrarily small length.

	\begin{theorem}\label{thm2}
		Let $S$ be a hyperbolic surface and $u\in T^1S$. If the geodesic ray $u(\R^+)$ meets an infinite sequence of closed hypercycles whose lengths tend to zero, then
		\begin{enumerate}
			\item $H^{ss} u\subsetneq W^{ss}u$,
			\item the set $W^{ss}u$ is an uncountable union of stable horocycles $H^{ss} v_i$.
		\end{enumerate}
	\end{theorem}

	These three cases are exhaustive: if the asymptotic fineness of $u(\R^+)$ vanishes and $u(\R^+)$ does not diverge through a cusp, then at least one of them occurs (Proposition \ref{trichotomy}). However, they are not mutually exclusive.



	The proof of Theorem \ref{thm2} follows the strategy of \cite{BurniolDalboHerrero2026} and \cite{BurniolDalbo2026}. We adapt to closed hypercycles the notion of \emph{winding} used in the previous works. One novelty is the use of the hyperbolic band model, in which the hypercycles around a given geodesic become horizontal lines, to obtain the necessary estimates.

	The paper is organized as follows. In Section \ref{sec:prelim} we recall the facts of hyperbolic geometry that will be used. In Section \ref{sec:fineness} we prove the trichotomy (Proposition \ref{trichotomy}), we deduce Theorem \ref{thm1} from Theorem \ref{thm2} and the previous works, and we give an example of the third case. In Section \ref{sec:winding} we introduce the winding around a closed hypercycle and we prove the Key Proposition (Proposition \ref{key_prop}), which controls the distance between a geodesic ray and its winding. Section \ref{sec:proof} is devoted to the proof of Theorem \ref{thm2}, where we construct uncountably many vectors in $W^{ss}u\setminus H^{ss}u$.

	\section{Preliminaries}\label{sec:prelim}
	 
	Throughout the paper we use the upper half-plane model $\mathbb H=\{z=x+iy\in\mathbb C : y>0\} $
	equipped with the metric
	\[
	ds^2=\frac{dx^2+dy^2}{y^2}.
	\]
	The Riemannian distance induced by this metric is denoted by $d$. We denote the unit tangent bundle of $\H$ by $T^1\mathbb H$, and the basepoint projection by $\pi:T^1\mathbb H\to\mathbb H$.
	
	\medskip
	
	\subsection{Horocycles and boundary at infinity}
	
	The boundary at infinity of $\mathbb H$ is $\partial\mathbb H=\mathbb R\cup\{\infty\}$.
	Every geodesic ray in $\mathbb H$ converges to a unique point of $\partial\mathbb H$, and two rays define the same point at infinity if and only if they remain at bounded distance from each other.
	
	\medskip
	
	Let $\xi\in\partial\mathbb H$, and let $(c(t))_{t\ge0}$ be a geodesic ray converging to $\xi$. 
	The \emph{Busemann function} based at $\xi$ is defined by
	\[
	B_\xi(x,y)
	=
	\lim_{t\to+\infty}
	\big(
	d(x,c(t)) - d(y,c(t))
	\big),
	\qquad x,y\in\mathbb H.
	\]
	This limit exists and does not depend on the choice of the geodesic ray $c$ asymptotic to $\xi$. 
	For fixed $x\in\mathbb H$, the level sets of the map $y\mapsto B_\xi(y,x)$ are called \emph{horocycles} centered at $\xi$. Geometrically, horocycles are circles tangent to the real line or horizontal lines.
	
	\medskip
	
	For $u\in T^1\mathbb H$, we denote by $u(+\infty)\in\partial\mathbb H$ its forward endpoint. 
	The \emph{stable horocycle} through $u$ is the subset
	\[
	H^{ss} u
	=
	\{ v\in T^1\mathbb H : v(+\infty)=u(+\infty) 
	\text{ and } 
	B_{u(+\infty)}(\pi(v),\pi(u))=0 \}. 
	\] We will also denote 
	\[
	H^{ws} u
	=
	\{ v\in T^1\mathbb H : v(+\infty)=u(+\infty)  \}.
	\]
	The following lemma can be proved by elementary computations using the formula for the hyperbolic distance \cite[Theorem 7.2.1]{Beardon1983}.
	\begin{lemma}\label{decreasing}
		Let $u,v\in T^1\H$ such that $v\in H^{ws} u$. Then $d(v(t),u(t))$ is non-increasing in $t\in \reals$.

	\end{lemma}
	
	\medskip
	
	\subsection{Distance on the unit tangent bundle}

	
	We consider the distance $d_1$ on the unit tangent bundle of $T^1\H$ defined as follows \cite{ballmannlectures}:
	$$ d_1(\tilde v,\tilde w):=d(\tilde v(0),\tilde w(0))+d(\tilde v(1),\tilde w(1)), \quad \tilde v, \tilde w \in T^1 \H.$$
	
	\begin{proposition}
		The distance $d_1$ is equivalent to the Sasaki distance on the unit tangent bundle $T^1 \H$. 
	\end{proposition}

	\medskip
	
	\subsection{Fuchsian groups}
	Throughout the paper, $\Gamma$ denotes a torsion-free Fuchsian group acting properly discontinuously by isometries on $\mathbb H$. 
	We consider the quotient surface $S=\Gamma\backslash\mathbb H$
	and its unit tangent bundle
	$	T^1S=\Gamma\backslash T^1\mathbb H$.
	We use a tilde to denote lifts to the universal cover: given $u\in T^1S$, we write $\tilde u\in T^1\mathbb H$ for a unit tangent vector projecting onto $u$ under the canonical projection $T^1\mathbb H\to T^1S$.
	
	The sets $H^{ss}u$ and $H^{ws}u$ are $\Gamma$-equivariant and therefore descend to well-defined subsets in $T^1S$, which we still denote by $H^{ss}u$ and $H^{ws}u$.
	
	On the unit tangent bundle $T^1 S$ of a hyperbolic surface $S=\Gamma\backslash\mathbb H $, the distance $d_1$ descends as follows: 
	
	$$d_1(u,v)=\inf_{\gamma\in \Gamma}\{d_1(\tilde u,\gamma\tilde v)\},$$
	where $\tilde u$ and $\tilde v$ are lifts of $u$ and $v$ to $T^1\H$. This distance is again equivalent to the Sasaki metric on $T^1S$.
	
	The strong-stable set of a vector $u\in T^1S$ is defined as
	$$
	W^{ss}u=\{ v\in T^1S:~ \lim_{t\to +\infty} d_1(g_tu,g_tv)=0 \}.
	$$ 
	We remark that, since the distance $d_1$ is equivalent to the Sasaki distance, they give the same notion of strong-stable set.

	\subsection{Hypercycles and hyperbolic isometries}\label{sec:hypercycles}

	Let $\gamma$ be a hyperbolic isometry of $\H$. It fixes exactly two points in $\partial\H$, the repelling fixed point $\gamma^-$ and the attracting fixed point $\gamma^+$. The geodesic $c$ joining $\gamma^-$ to $\gamma^+$ is the \emph{axis} of $\gamma$, and $\gamma$ acts on $c$ as a translation of length
	\[
	\ell(\gamma)=\min_{z\in\H} d(z,\gamma z).
	\]
	The displacement of a point $z\in\H$ by $\gamma$ only depends on its distance to the axis \cite[Theorem 7.35.1]{Beardon1983}:
	\begin{equation}\label{eq:displacement}
	\sinh\left(\frac{d(z,\gamma z)}{2}\right)=\cosh\bigl(d(z,c)\bigr)\,\sinh\left(\frac{\ell(\gamma)}{2}\right).
	\end{equation}
	In particular, the displacement is minimal along the axis and increases with the distance to the axis.

	For $r>0$, the set of points of $\H$ at distance $r$ from $c$ consists of two curves, one on each side of $c$, called \emph{hypercycles} of $\gamma$. Each of them is invariant under $\gamma$. In the upper half-plane model, they are arcs of Euclidean circles, or Euclidean half-lines, joining $\gamma^-$ to $\gamma^+$. Given a point $q\in\H$, we denote by $\mathcal H(\gamma,q)$ the hypercycle of $\gamma$ through $q$.

	\begin{proposition}\label{closest_points}
		Let $c_1$, $c_2$ be two geodesics in $\mathbb{H}$ such that $d(c_1,c_2)>0$. Then there exist unique $p = c_1(t_1)$, $q=c_2(t_2)$ realizing the distance between $c_1$ and $c_2$. Moreover, there is an hypercycle at distance $d(c_1,c_2)$ from $c_1$ that is tangent to $c_2$ at $q$.
	\end{proposition}

	If $\gamma$ belongs to $\Gamma$, the hypercycle $\mathcal H(\gamma,q)$ projects to a closed curve in $S=\Gamma\backslash\H$, which we call a \emph{closed hypercycle}. A direct computation shows that the arc of $\mathcal H(\gamma,q)$ between $q$ and $\gamma q$ has length
	\begin{equation}\label{eq:arc}
	\ell(\gamma)\cosh d(q,c).
	\end{equation}
	If $\gamma$ is primitive in $\Gamma$, this is the length of the closed hypercycle of $\gamma$.

	\subsection{The band model}\label{sec:band}

	For some computations it is convenient to use the band model
	\[
	\mathbb B=\left\{ x+iy\in\mathbb C \ :\ -\frac{\pi}{2}<y<\frac{\pi}{2}\right\}, \qquad ds^2=\frac{dx^2+dy^2}{\cos^2 (y)}.
	\]
	The map $z\mapsto ie^{z}$ is an isometry from $\mathbb B$ onto $\H$, which sends the real line $\{y=0\}$ onto the imaginary axis. The boundary at infinity of $\mathbb B$ consists of the two lines $\{y=\pm\pi/2\}$ together with the two points $\pm\infty$. We will use the following facts.
	\begin{enumerate}
		\item The real line $\{y=0\}$ is a geodesic. For $\lambda>0$, the translation $z\mapsto z+\lambda$ is a hyperbolic isometry with axis $\{y=0\}$ and translation length $\lambda$.
		\item The vertical segments $\{x=k\}$, $k\in\reals$, are geodesics orthogonal to $\{y=0\}$.\label{band:orth}
		\item The horizontal lines $\{y=k\}$, $k\in(-\pi/2,\pi/2)\setminus\{0\}$, are hypercycles. More precisely, for $x\in\reals$ and $y\in(-\pi/2,\pi/2)$,
		\[
		\cosh d(x,x+iy)=\frac{1}{\cos (y)}.
		\]
		\item \textbf{Equations of the geodesics relative to $\{y=0\}$} (see Figure \ref{fig:band}). Every geodesic other than $\{y=0\}$ is of one of the following three types.
		\begin{enumerate}[(a)]
			\item \emph{Ultraparallel geodesics.} A geodesic at positive distance from $\{y=0\}$ has an equation of the form\label{band:same}
			\[
			\cosh(x-x_0)=\alpha\sin (y), \qquad x_0\in\reals,\ |\alpha|>1.
			\]
			Its endpoints are $x_0\pm r+i\varepsilon\pi/2$, where $\varepsilon$ is the sign of $\alpha$ and $\cosh r=|\alpha|$, and its point closest to $\{y=0\}$ is $x_0+iy_0$, where $\sin y_0=1/\alpha$.
			\item \emph{Intersecting geodesics.} A geodesic intersecting $\{y=0\}$ has an equation of the form\label{band:cross}
			\[
			\sinh(x-x_0)=\beta\sin (y), \qquad x_0\in\reals,\ \beta\in\reals.
			\]
			Its endpoints are $x_0+s+i\pi/2$ and $x_0-s-i\pi/2$, where $\sinh s=\beta$. It meets $\{y=0\}$ at $x_0$, and the angle $\theta\in(0,\pi)$ from the positive direction of $\{y=0\}$ to the geodesic satisfies $\cot\theta=\beta$. The case $\beta=0$ gives the geodesics of item (\ref{band:orth}).
			\item \emph{Asymptotic geodesics.} A geodesic asymptotic to $\{y=0\}$ has an equation of the form\label{band:infty}
			\[
			\sin(y)=\varepsilon e^{-\delta(x-x_0)}, \qquad x_0\in\reals,\ \varepsilon,\delta\in\{-1,1\}.
			\]
			Its endpoints are $x_0+i\varepsilon\pi/2$ and $+\infty$ if $\delta=1$, or $-\infty$ if $\delta=-1$.
		\end{enumerate}
	\end{enumerate}

	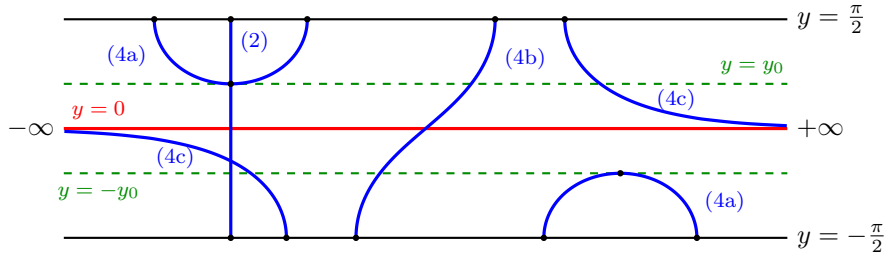
\begin{figure}[htbp]
	\centering
	\begin{tikzpicture}[scale=0.92,line join=round]

	  \draw[thick] (-5.2,pi/2) -- (5.2,pi/2) node[right] {$y=\frac{\pi}{2}$};
	  \draw[thick] (-5.2,-pi/2) -- (5.2,-pi/2) node[right] {$y=-\frac{\pi}{2}$};
	  \node[left] at (-5.2,0) {$-\infty$};
	  \node[right] at (5.2,0) {$+\infty$};

	  \draw[thick,dashed,green!55!black] (-5.2,0.6425) -- (5.2,0.6425);
	  \draw[thick,dashed,green!55!black] (-5.2,-0.6425) -- (5.2,-0.6425);
	  \node[green!55!black,above,font=\footnotesize] at (4.7,0.6425) {$y=y_0$};
	  \node[green!55!black,below,font=\footnotesize] at (-4.7,-0.6425) {$y=-y_0$};

	  \draw[very thick,red] (-5.2,0) -- (5.2,0);
	  \node[red,above,font=\footnotesize] at (-4.7,0) {$y=0$};

	  \draw[very thick,blue,domain=-90:90,samples=80,smooth,variable=\t]
	    plot ({-2.8+1.1*sin(\t)},{rad(asin(min(1,cosh(1.1*sin(\t))/cosh(1.1))))});
	  \draw[very thick,blue,domain=-90:90,samples=80,smooth,variable=\t]
	    plot ({2.8+1.1*sin(\t)},{-rad(asin(min(1,cosh(1.1*sin(\t))/cosh(1.1))))});

	  \draw[very thick,blue] (-2.8,-pi/2) -- (-2.8,pi/2);

	  \draw[very thick,blue,domain=-90:90,samples=60,smooth,variable=\t]
	    plot ({ln(1.1752*sin(\t)+sqrt(1.1752*sin(\t)*1.1752*sin(\t)+1))},{rad(\t)});

	  \draw[very thick,blue,domain=2.34:90,samples=80,smooth,variable=\t]
	    plot ({2-ln(sin(\t))},{rad(\t)});
	  \draw[very thick,blue,domain=2.34:90,samples=80,smooth,variable=\t]
	    plot ({-2+ln(sin(\t))},{-rad(\t)});

	  \node[blue,font=\footnotesize] at (-4.3,1.05) {(\ref{band:same})};
	  \node[blue,font=\footnotesize] at (4.3,-1.05) {(\ref{band:same})};
	  \node[blue,right,font=\footnotesize] at (-2.8,1.25) {(\ref{band:orth})};
	  \node[blue,right,font=\footnotesize] at (1.0,1.0) {(\ref{band:cross})};
	  \node[blue,font=\footnotesize] at (3.6,0.42) {(\ref{band:infty})};
	  \node[blue,font=\footnotesize] at (-3.6,-0.42) {(\ref{band:infty})};

	  \fill (-2.8,0.6425) circle (1.3pt);   
	  \fill (2.8,-0.6425) circle (1.3pt);
	  \fill (-3.9,pi/2) circle (1.3pt);  \fill (-1.7,pi/2) circle (1.3pt);
	  \fill (1.7,-pi/2) circle (1.3pt);  \fill (3.9,-pi/2) circle (1.3pt);
	  \fill (-2.8,pi/2) circle (1.3pt);  \fill (-2.8,-pi/2) circle (1.3pt);
	  \fill (1,pi/2) circle (1.3pt);     \fill (-1,-pi/2) circle (1.3pt);
	  \fill (2,pi/2) circle (1.3pt);     \fill (-2,-pi/2) circle (1.3pt);

	\end{tikzpicture}
	\caption{Geodesics of the band model $\mathbb B$. The labels refer to the items of Section \ref{sec:band}. The geodesics of item (\ref{band:same}) are tangent to the hypercycles $\{y=\pm y_0\}$ at their points closest to $\{y=0\}$, and the vertical geodesic through one of these points is the common perpendicular. The geodesics of item (\ref{band:infty}) are asymptotic to $\{y=0\}$.}
	\label{fig:band}
	\end{figure}

	
	




%

\par 

\section{Vanishing asymptotic fineness}\label{sec:fineness}

Given $x\in S$, we denote by $\mathrm{Inj}(x)$ the injectivity radius of $S$ at $x$. The next proposition reduces the study of the vectors whose geodesic ray has vanishing injectivity radius, without diverging through a cusp, to three cases. Each case is governed by the geometry of the thin parts visited by the ray: short closed geodesics, cusps, or short geodesic loops enclosed by hypercycles.

\begin{proposition}\label{trichotomy}
	Let $S=\Gamma\backslash\H$ be a hyperbolic surface and let $u\in T^1S$. Assume that the geodesic ray $u(\reals^+)$ does not diverge through a cusp and that its injectivity radius is not bounded below, that is,
	\[
	\inf_{t\ge 0}\mathrm{Inj}\bigl(u(t)\bigr)=0.
	\]
	Then at least one of the following holds:
	\begin{enumerate}
		\item $u(\reals^+)$ crosses an infinite sequence of closed geodesics whose lengths tend to $0$,
		\item $u(\reals^+)$ intersects an infinite sequence of closed horocycles whose lengths tend to $0$,
		\item $u(\reals^+)$ intersects an infinite sequence of closed hypercycles whose lengths tend to $0$ but it does not meet the associated closed geodesics.
	\end{enumerate}
\end{proposition}

\begin{proof}
	Since $\inf_{t\ge 0}\mathrm{Inj}(u(t))=0$, there exists a sequence of times $t_n\ge 0$ with $\mathrm{Inj}(u(t_n))\to 0$.

	Fix a lift $\tilde u\in T^1\H$ of $u$. The injectivity radius at $u(t_n)$ equals half the length of the shortest geodesic loop based at $u(t_n)$, so there exists $\gamma_n\in\Gamma\setminus\{\mathrm{id}\}$ such that
	\[
	\varepsilon_n:=d\bigl(\tilde u(t_n),\gamma_n\tilde u(t_n)\bigr)=2\,\mathrm{Inj}\bigl(u(t_n)\bigr)\longrightarrow 0.
	\]
	Since $\Gamma$ is torsion free, each $\gamma_n$ is either parabolic or hyperbolic. After passing to a subsequence, either all the $\gamma_n$ are parabolic, or all of them are hyperbolic.

	\medskip
	\noindent\textbf{Parabolic case.}
	Assume that every $\gamma_n$ is parabolic. Let $x_n\in\partial\H$ be its fixed point and let $\tilde H_n$ be the horocycle centered at $x_n$ passing through $\tilde u(t_n)$. Denote by $\ell(\tilde H_n,\gamma_n)$ the length of the horocyclic arc of $\tilde H_n$ between $\tilde u(t_n)$ and $\gamma_n\tilde u(t_n)$. A direct computation in the upper half-plane model shows that
	\[
	\sinh\!\left(\frac{\varepsilon_n}{2}\right)=\frac{\ell(\tilde H_n,\gamma_n)}{2},
	\]
	hence $\ell(\tilde H_n,\gamma_n)\to 0$. The horocycle $\tilde H_n$ projects to a closed horocycle of $S$ of length at most $\ell(\tilde H_n,\gamma_n)$ passing through $u(t_n)$. So $u(\reals^+)$ intersects a sequence of closed horocycles whose lengths tend to $0$. As $u(\reals^+)$ does not diverge through a cusp, this shows that $u$ is cusp-recurrent in the sense of \cite{BurniolDalbo2026}, and we are in case (2).

	\medskip
	\noindent\textbf{Hyperbolic case.}
	Assume that every $\gamma_n$ is hyperbolic. Let $c_n$ be its axis and $\ell(\gamma_n)$ its translation length.
	Clearly, $\ell(\gamma_n)\le d(\tilde u(t_n), \gamma_n \tilde u(t_n))=\varepsilon_n$ tends to $0$ when $n\to \infty$.

	We distinguish two subcases.

	\medskip
	\emph{Subcase (i): $u(\reals^+)$ meets the closed geodesic obtained by projecting $c_n$ onto $S$ for infinitely many $n$.}
	Passing to a subsequence, this holds for every $n$. These closed geodesics have length at most $\ell(\gamma_n)$. Since $\ell(\gamma_n)\to 0$, we are in case (1).

	\medskip
	\emph{Subcase (ii): $u(\reals^+)$ meets the closed geodesic obtained by projecting $c_n$ onto $S$ for at most finitely many $n$.}
	Passing to a subsequence, $u(\reals^+)$ does not meet these closed geodesics for any $n$. The hypercycle $\mathcal H(\gamma_n,\tilde u(t_n))$ projects to a closed hypercycle in $S$ passing through $u(t_n)$. Let $\gamma_n'$ be a generator of the stabilizer of $c_n$ in $\Gamma$, so that $\ell(\gamma_n')\le\ell(\gamma_n)$. By \eqref{eq:arc}, the length of the projection of $\mathcal H(\gamma_n,\tilde u(t_n))$ is 
	\[
	\ell(\gamma_n')\cosh(d(\tilde u(t_n),c_n)).
	\]
	Using the inequality $x\le 2\sinh(x/2)$ and \eqref{eq:displacement}, we obtain
	\[
	\ell(\gamma_n')\cosh(d(\tilde u(t_n),c_n))\le\ell(\gamma_n)\cosh(d(\tilde u(t_n),c_n))\le   2\sinh\left(\frac{\varepsilon_n}{2}\right),
	\]
	which tends to $0$ when $n\to +\infty$. 
	Hence we are in case (3).
	



\end{proof}

We now deduce Theorem \ref{thm1} from Proposition \ref{trichotomy}, Theorem \ref{thm2} and the previous works.

\begin{myproof}{Theorem}{\ref{thm1}}
Let $S$ be a hyperbolic surface and $u\in T^1S$. Since $\mathrm{Inj}$ is continuous and positive on $S$, the asymptotic fineness of $u(\R^+)$ vanishes if and only if $\inf_{t\ge 0}\mathrm{Inj}(u(t))=0$.

Assume first that the asymptotic fineness of $u(\R^+)$ is positive. Then Bellis proved that $H^{ss}u=W^{ss}u$ \cite[Th\'eor\`eme D]{BellisThesis2018}. If the geodesic ray $u(\R^+)$ diverges through a cusp, then $H^{ss}u$ is a closed horocycle around that cusp and the equality $H^{ss}u=W^{ss}u$ is clear.

Assume now that the asymptotic fineness of $u(\R^+)$ vanishes and that $u(\R^+)$ does not diverge through a cusp. By Proposition \ref{trichotomy}, at least one of the following cases occurs.
\begin{enumerate}
	\item The ray $u(\R^+)$ crosses closed geodesics whose lengths tend to $0$. The conclusion follows from \cite[Main Theorem]{BurniolDalboHerrero2026}.
	\item The ray $u(\R^+)$ intersects closed horocycles whose lengths tend to $0$, so $u$ is cusp-recurrent and the conclusion follows from \cite[Theorem 2]{BurniolDalbo2026}.
	\item The ray $u(\R^+)$ intersects closed hypercycles whose lengths tend to $0$. The conclusion follows from Theorem \ref{thm2}.
\end{enumerate}
In all cases, $H^{ss}u\subsetneq W^{ss}u$ and $W^{ss}u$ is an uncountable union of stable horocycles.
\end{myproof}

The following example shows that case (3) of Proposition \ref{trichotomy} can occur while cases (1) and (2) do not.

\begin{example}\label{example_case3}
	There exist a hyperbolic surface $S$ without cusps and a vector $u\in T^1S$ such that
	\begin{itemize}
		\item $\liminf_{t\ge 0}\mathrm{Inj}\bigl(u(t)\bigr)=0$,
		\item for every sequence of closed geodesics $(c_k)$ with $\ell(c_k) \to 0$, we have $d\bigl(c_k,u(\R^+)\bigr)\to +\infty$.
	\end{itemize}

	Let $\ell_n$ be a sequence of positive numbers converging to $0$ and let $D>0$. Let $P_n$ be the pair of pants with boundary geodesic components of lengths $\ell_n$, $D$ and $D$ \cite[Chapter~3]{Buser1992}. For $n\ge 0$, $P_n$ is glued to $P_{n+1}$ through a geodesic boundary of length $D$. Then funnels are attached to the remaining boundary components. We thus obtain a complete hyperbolic surface with no cusps called a flute surface.
	
	Let $\alpha_n$ be the boundary component of $P_n$ of length $\ell_n$. 
	Let $k_n$ be a sequence of integers. We consider a vector $u\in T^1S$ whose geodesic ray $u(\R^+)$ has the following behavior:
	\begin{itemize}
		\item it starts in $P_0$, that is, $u(0)\in P_0$,
		\item it visits the pairs of pants $P_0, P_1, P_2,\dots $ in this order,
		\item for every $n\ge 1$, in $P_n$ the ray turns $k_n$ times around $\alpha_n$ and does not turn around the other two boundary geodesics.
	\end{itemize}

	By choosing the numbers of turns $k_n$ appropriately, we can guarantee that the distance between $u(\R^+)$ and $\alpha_n$ tends to infinity, while $u(\R^+)$ visits points of the collar of $\alpha_n$ where the injectivity radius is small, so that the injectivity radius along the ray vanishes.

	To conclude one shows that there exists a constant $\varepsilon_0>0$ such that every closed geodesic of $S$ distinct from the geodesics $\alpha_n$ has length at least $\varepsilon_0$. 
	
\end{example}


\section{Winding around a closed hypercycle}\label{sec:winding}

\begin{definition}
	Let $\gamma$ be an hyperbolic isometry of $\mathbb{H}$ with axis $c$ and let $\alpha: \reals \to \mathbb{H}$ be a unit speed parametrization of a hypercycle of $c$. The orientation of $\alpha $ is positive with respect to $\gamma$ if there exists $D>0$ such that for all $t\in \reals$,
	\[
	\gamma \alpha (t) = \alpha ( t+ D).
	\]
\end{definition}

\begin{definition}
	Let $\gamma$ be an hyperbolic isometry of $\mathbb{H}$ with axis $c$ an let $\tilde u\in T^1 \mathbb{H}$ such that $d(c,\tilde u(\reals))>0$. Let $q=\tilde u(t_1)$ be the point in $\tilde u(\reals)$ that realizes the distance to $c$. Let $\alpha: \reals \to \mathbb{H}$ be the positive unit speed parametrization of the hypercycle of $c$ through $q$ such that $\alpha(0)=q$. We say that $\tilde u$ is in \emph{good position} with respect to $\gamma$ if $t_1>0$ and 
	\[
	\frac{\mathrm{d}\tilde u}{\mathrm{d} t}(t_1) =\frac{\mathrm{d}\alpha}{\mathrm{d} t}(0).
	\]
\end{definition}

\begin{definition}
	Let $\tilde u\in T^1 \H$ be in good position with respect to a hyperbolic isometry $\gamma$. 
	We define the \emph{winding of $\tilde u$ around} $\gamma$ (see Figure \ref{pic1}), denoted by $\mathrm{Wind}_{\gamma}(\tilde{u})\in T^{1}\mathbb{H}$, as the unique vector satisfying
	\begin{itemize}
		\item $\mathrm{Wind}_{\gamma}(\tilde {u})(0)=\tilde{u}(0)$,
		\item $\mathrm{Wind}_{\gamma}(\tilde{u})(+\infty)= \gamma \, \tilde{u}(+\infty) $.
	\end{itemize}
\end{definition}

 


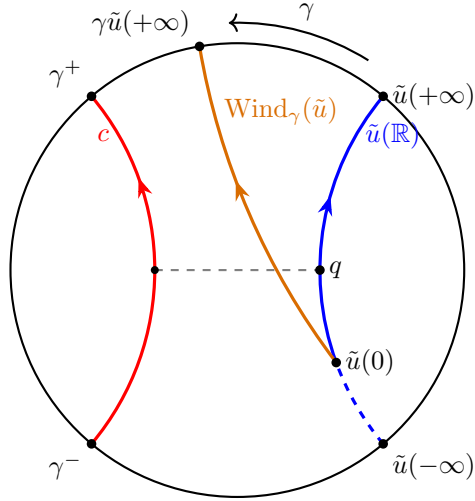
\begin{figure}[htbp]
\centering
\begin{tikzpicture}[scale=3,line join=round]

  \draw[thick] (0,0) circle (1);

  \draw[very thick,red,midarrow=0.75] (230:1) arc[start angle=-40,end angle=40,radius=1.19175];
  \node[red,left] at (-0.524,0.596) {$c$};

  \draw[very thick,blue,dashed] (-50:1) arc[start angle=220,end angle=200,radius=1.19175];
  \draw[very thick,blue,midarrow=0.6] (0.43584,-0.40760) arc[start angle=200,end angle=140,radius=1.19175];
  \node[blue,right] at (0.524,0.596) {$\tilde u(\mathbb R)$};

  \draw[dashed,thick,gray] (-0.36397,0) -- (0.36397,0);

  \draw[very thick,orange!85!black,midarrow=0.6] (0.43584,-0.40760) arc[start angle=-142.831,end angle=-170.400,radius=3.18617];
  \node[orange!85!black,right] at (-0.10,0.70) {$\mathrm{Wind}_\gamma(\tilde u)$};

  \draw[->,thick] (58:1.09) arc[start angle=58,end angle=92,radius=1.09];
  \node at (75:1.18) {$\gamma$};

  \fill (0.36397,0) circle (0.6pt) node[right] {$q$};
  \fill (-0.36397,0) circle (0.5pt);
  \fill (0.43584,-0.40760) circle (0.6pt) node[right] {$\tilde u(0)$};

  \fill (130:1) circle (0.6pt) node[above left] {$\gamma^+$};
  \fill (230:1) circle (0.6pt) node[below left] {$\gamma^-$};
  \fill (50:1) circle (0.6pt) node[right] {$\tilde u(+\infty)$};
  \fill (-50:1) circle (0.6pt) node[below right] {$\tilde u(-\infty)$};
  \fill (99.6:1) circle (0.6pt) node[above left] {$\gamma\tilde u(+\infty)$};

\end{tikzpicture}
\caption{The winding $\mathrm{Wind}_\gamma(\tilde u)$ of $\tilde u$ around $\gamma$, in the disk model.}
\label{pic1}
\end{figure}

\begin{definition}\label{def_winding_time}

The \emph{winding time} associated to $\gamma$ is the real number $\tau_{\gamma,\tilde u}$ defined by
$$
\tau_{\gamma,\tilde u}=B_{\tilde u(+\infty)}(\gamma^{-1}\tilde u(0), \tilde u(0)).
$$

\end{definition}


 \begin{proposition}[bound of the winding time]\label{bound_wind}
	Let $\tilde u\in T^1 \mathbb{H}$ be in good position with respect to a hyperbolic isometry $\gamma$. Let $q=\tilde u(t_1)$ be the point in $\tilde u(\reals)$ that realizes the distance to the axis $c$ of $\gamma$.  	
 	Then
 	\[
 	\left| \tau_{\gamma,\tilde u} \right| \leq d(q, \gamma q).
 	\]
 	
 \end{proposition}
 \begin{proof}
 	As $q$ belongs to $\tilde u(\R^+)$, we have
 	$$
 	B_{\tilde u(+\infty)}(q,\tilde u(0))=-d(q,\tilde u(0)).
 	$$
 	Thus
 	\begin{align*}
 		B_{\tilde u(+\infty)}(\gamma^{-1}\tilde u(0),& \tilde u(0)) = \\ 
 		& =  B_{\tilde u(+\infty)}(\gamma^{-1}\tilde u(0),\gamma^{-1}q)+  
 		B_{\tilde u(+\infty)}(\gamma^{-1}q,q) +B_{\tilde u(+\infty)}(q,\tilde u(0)) \\
 		& \leq  d(\gamma^{-1}\tilde u(0),\gamma^{-1}q)+d(\gamma^{-1} q,q) -d(q,\tilde u(0))\\
 		& = d( q, \gamma q).
 	\end{align*}
 	
 	For the other inequality, we observe that $\gamma \tilde u (+\infty)$ is between $\tilde u (+ \infty)$ and the attrating fixed point $x_\gamma^+$ of $\gamma$. Then the the geodesic ray $[\tilde u(0), \gamma \tilde u(+\infty))$ is closer than $\tilde u (\reals^+)$ to the axis $c$ of $\gamma$. Let $q'$ be the point of $[\tilde u(0), \gamma \tilde u(+\infty))$ closest to $c$. We have $d(q',c)\le d(q,c)$. Then $d(q',\gamma q')\le d(q, \gamma q)$.

 	Since $q'$ is in the geodesic ray $[\tilde u(0), \gamma \tilde u(+\infty))$, we have
	\[
	B_{\tilde u(+\infty)}(\gamma ^{-1}\tilde u(0), \gamma ^{-1} q') =
	B_{\gamma \tilde u(+\infty)}(\tilde u(0),q')=d(\tilde u(0),q'). 
	\]

 	Therefore 	
 	\begin{align*}
 		B_{\tilde u(+\infty)} &(\gamma^{-1}\tilde u(0), \tilde u(0)) = \\ 
		&= B_{\tilde u(+\infty)}(\gamma^{-1}\tilde u(0),\gamma ^{-1}q')+B_{\tilde u(+\infty)}(\gamma ^{-1} q',q')+B_{\tilde u(+\infty) }( q',\tilde u(0)) \\
 		& \geq d(\tilde u(0),q')-d(\gamma ^{-1} q', q')-d( q',\tilde u(0))\\
 		&=  -d( q', \gamma q') \ge -d( q, \gamma q).
 	\end{align*}
 \end{proof}
 
 \begin{proposition}[Key Proposition]\label{key_prop} 
 	Let $\tilde u\in T^1 \H$ be in good position with respect to a hyperbolic isometry $\gamma$. Let $q$ be the point in $\tilde u(\reals)$ closest to the axis of $\gamma$. Assume that $d(q,\gamma q)\le 1$.	
 	Then, for every $t\ge 0$,
 	\begin{align*}
 		\min \{ d_1\!\left(g_{t+\tau_{\gamma ,\tilde u}}\bigl(\operatorname{Wind}_{\gamma}(\tilde u)\bigr),\, g_t\tilde u\right),
 		d_1\!\left(g_{t+\tau_{\gamma ,\tilde u}}\bigl(\operatorname{Wind}_{\gamma}(\tilde u)\bigr),\, \gamma g_t\tilde u\right)
 		\}
 		\\ \le 17 d(q,\gamma q).
 	\end{align*}
 	
 \end{proposition}
 \begin{proof}
 	For simplicity, we denote $\tilde v=\mathrm{Wind}_{\gamma}(\tilde u)$ and $\tau=\tau_{\gamma ,\tilde u}$.
 	Let $q=\tilde u(t_1)$, $t_1\ge 0$, be the point that realizes the distance of $\tilde u(\reals)$ to the axis of $\gamma$. We observe that $\tilde u(0)$ and $\gamma \tilde u (+\infty)$ lie in the same half-plane in $\mathbb H \setminus (\gamma^-,\gamma^+)$. Hence, the geodesic generated by $\tilde v$ does not intersect the axis of $\gamma$. Let $q' = \tilde v (t_1')$, $t_1'\in \reals$, be the point that realizes the distance of $\tilde v(\reals)$ to the axis of $\gamma$. The first step is to control the distance between $q$ and $q'$.
 	
 	To do the next computations we use the band model $\mathbb B$ (see Section \ref{sec:band}). We can assume that the axis of $\gamma$ is the geodesic $y=0$ and that $\gamma$ has the form $\gamma(z)= z+\lambda$, where $\lambda>0$ is the translation length of $\gamma$. 
	
	
	The geodesics at positive distance from $y=0$ are those of item (\ref{band:same}). 
	Without loss of generality, we can assume that the equation of $\tilde u(\reals)$ has $x_0=0$ and $\alpha>1$. Hence,
	\[
	\tilde u(\reals) = \{ (x,y) \in \reals \times (-\pi/2,\pi/2) \, | \, \cosh (x) = \alpha \sin (y) \},
	\]
	which lies in the upper half-strip $y>0$. The endpoints at infinity of $\tilde u (\reals)$ are
	\[
	\tilde u(-\infty) = -r + \pi/2 i, \quad \tilde u(+\infty) = r + \pi/2 i,
	\] where $r>0$ satisfies $\cosh(r)=\alpha$. The point $q$ realizing the distance to the axis $y=0$ is $q= y_0 i$ with $\alpha = 1/\sin(y_0)$. The initial point $\tilde u(0)$ lies between $\tilde u(-\infty)$ and $q$, so it has negative real part.
	
	The geodesic $\tilde v(\reals)$ is also at positive distance from the axis $y=0$ and in the upper half-strip.
	Let \[
	\cosh(x-x_0') = \alpha ' \sin (y), \quad x_0'\in \reals, \, \alpha ' >1,
	\]
	be the equation of $\tilde v(\reals)$. The endpoints of $\tilde v(\reals )$ are 
	\[
	\tilde v (-\infty ) = x_0'-r' + \pi/2 i, \quad \tilde v(+\infty) = x_0'+r' + \pi/2 i,
	\]
	where $r'>0$ satisfies $\cosh(r')=\alpha'$. The point $q'$ realizing the distance to the axis is $q' = x_0' + y_0' i$ with $\alpha' = 1/\sin(y_0')$.

	By definition, $\tilde v(+\infty) = \gamma \tilde u(+\infty)$, which implies $ r + \lambda = x_0' +r'$. The point $\tilde v(-\infty)$ must lie between $-r +\pi/2 i$ and $\pi/2 i$. 
	
	
 	
 	From the position of the endpoints of $\tilde u(\reals)$ and $\tilde v(\reals)$, we obtain the following estimates 
 	\begin{enumerate}
 		\item $0<y_0'<y_0$,\label{eq1}
 		\item $0<r'-r<\lambda/2$, \label{eq2}
 		\item $0<x_0' < \lambda$.\label{eq3}
 	\end{enumerate}

	Consider the function $f(t)= \arcsin(1/\cosh(t))$ defined for $t\in \reals$. We have $f(r)=y_0$ and $f(r')=y_0'$. An easy computation shows that $f$ is $1$-Lipschitz, hence 
	\[
	|y_0- y_0'|\le |r-r'| <\frac{\lambda}{2}.
	\]

 	Introducing $q''= x_0' + y_0 i$, we have
 	\[
 	d(q,q')\le d(q, q'' ) + d(q'',q').
 	\]
 	
 	First, we notice that $\gamma q = \lambda + y_0 i$. Since $0<x_0'<\lambda $, the point $q''= x_0' + y_0 i$ belongs to the hypercycle segment between $q$ and $\gamma q$. This implies that $d(q,q'')\le d(q,\gamma q)$.
 	
 	The points $q''$ and $q'$ lie on the vertical geodesic $\{x=x_0'\}$ (item (\ref{band:orth})). Hence $d(q'',q')$ is the length of the vertical segment joining them and, since $0<y_0'<y_0$,
 	\[
 	d( q'', q' )=\int_{y_0'}^{y_0}\frac{dt}{\cos(t)}\le \frac{y_0-y_0'}{\cos(y_0)}\le \frac{\lambda/2}{\cos(y_0)}.
 	\]
 	
 	By item (3), $\cosh(d(0,q))=1/\cos(y_0)$. Since $d(0,q)$ is the distance from $q$ to the axis of $\gamma$, \eqref{eq:displacement} gives
 	\[
 	\sinh (\frac{d(q,\gamma q)}{2}) = \cosh(d(0,q)) \sinh(\frac{\lambda}{2}).
 	\]
 	These together imply
 	 	\[
 	 	d( q'', q' )\le  \frac{\lambda/2}{\cos(y_0)} \le \sinh(\lambda/2) \cosh (d(0,q)) = \sinh (\frac{d(q,\gamma q)}{2}).
 	 	\]
 	If $0\le x\le 1$ then $\sinh(x/2)\le x$. Since $d(q,\gamma q)\le 1$, we obtain
 	\[
 	d(q,q') \le 2 d(q, \gamma q).
 	\]

  	
 	The rest of the proof follows from the bound of the distance between $q$ and $q'$. Recall that $q=\tilde u (t_1)$ and $q'=\tilde v (t'_1)$. By the triangular inequality,
 	\[
 	|t_1 - t'_1 | \le d(q,q' ) \le 2 d(q,\gamma q).
 	\]
 	
 	We now bound the distance between the basepoints of the geodesics of $\tilde u$ and $\tilde v$.
 	
 	\noindent\textbf{Case 1.}  If $0\le t\le t_1$, then 
 	\[
 	d\bigl(\tilde v(t),\tilde u(t)\bigr)\le 4 d(q,\gamma q).
 	\] 

 Since the distance in $\mathbb{H}$ between two geodesic rays starting at the same point is increasing, it follows 
 \[
  	\forall\,0\le t\le t_1,\qquad
  	d\bigl(\tilde v(t),\tilde u(t)\bigr)\le d\bigl(\tilde v(t_1),\tilde u(t_1)\bigr).
  	\]
 
Using the bounds established above, one obtains
 	\[
 	d\bigl(\tilde v(t_1),\tilde u(t_1)\bigr) \le 
 	|t_1 - t'_1 | + d\bigl(\tilde v(t'_1),\tilde u(t_1)\bigr) \le  2 d\bigl(q',q\bigr) \le 4 d(q,\gamma q),
  	\]
which implies the statement.

 	\smallskip
 	\noindent\textbf{Case 2:} If $ t\ge  t_1$, then 
 	\[
 	d\bigl(\tilde v(t), \gamma \tilde u(t)\bigr)\le 5d(q,\gamma q).
 	\] 
 	
 	Since $\tilde v$ and $\gamma \tilde u$ have the same point at infinity, the distance $d\bigl(\tilde v(t), \gamma \tilde u(t)\bigr)$ is decreasing (Lemma \ref{decreasing}), and hence for $t\ge t_1$,
 	\[
 	d\bigl(\tilde v(t), \gamma\tilde u(t)\bigr)
 	\le d\bigl(\tilde v(t_1),\gamma \tilde u(t_1)\bigr).
 	\]
 	Moreover,
 	\[
 	d\bigl(\tilde v(t_1), \gamma \tilde u(t_1)\bigr)
 	\le |t_1-t_1'| + d\bigl(\tilde v(t_1'), \gamma \tilde u(t_1) \bigr) \le |t_1-t_1'| + d(q', q) + d(q,\gamma q)
 	\]
 	By the bounds proved above, the first two terms on the right hand side are bounded by $2 d(q,\gamma q)$.
 	We conclude that, for $t\ge t_1$, $d\bigl(\tilde v(t), \gamma \tilde u(t)\bigr)  \le 5d(q,\gamma q)$. 
 	
 	\smallskip
 	
 Let us now compute the bounds for the distance $d_1$. Recall that 
 \[
d_1\!\left(g_{t} \tilde v ,\, g_t\tilde u\right) = d(\tilde v (t), \tilde u (t)) + d(\tilde v (t+1), \tilde u (t+1)). 
 \]
 We can distinguish three scenarios:
 
 \begin{enumerate}[(a)]
 	\item If $t \le t_1 -1 $, by the first case, we have that both
 	\[
 	d(\tilde v(t), \tilde u(t))
 	\quad \text{and} \quad
 	d(\tilde v(t+1), \tilde u(t+1))
 	\]
 	are bounded by $4 d(q,\gamma q)$. Then
 	\[
 	d_1(g_{t}\tilde v, g_t \tilde u) \le 8 d(q,\gamma q).
 	\]
 	
 	\item If $t \ge t_1$, by the second case, we can
 	bound both
 	\[
 	d(\tilde v(t), \gamma \tilde u(t))
 	\quad \text{and} \quad
 	d(\tilde v(t+1), \gamma \tilde u(t+1))
 	\]
 	by $5d(q,\gamma q)$. Then
 	\[
 	d_1(g_{t}\tilde v, \gamma g_t \tilde u) \le 10 d(q,\gamma q).
 	\]
 	
 	\item Let $t_1 -1 < t < t_1$. We have
 	\begin{align*}
 		d_1(g_{t}\tilde v, g_t \tilde u) &= d(\tilde v (t), \tilde u (t)) + d(\tilde v (t+1), \tilde u (t+1)) \\
 		&\le  d(\tilde v (t), \tilde u (t)) + d(\tilde v (t+1), \gamma \tilde u (t+1)) + d( \gamma \tilde u (t+1), \tilde u (t+1)).
 	\end{align*}
	The first and the second terms in the last line are bounded by $4d(q,\gamma q)$ and $5 d(q,\gamma q)$ by Cases 1 and 2 above, respectively. 
 	
	In order to bound the third term we apply \eqref{eq:displacement}. Writing $z= \tilde u(t+1)$ and $c$ for the axis of $\gamma$, we obtain
 	\begin{equation*}
 		\sinh\!\left(\frac{d(z, \gamma z)}{2}\right)
 		=
 		\cosh \bigl(d(z, c)\bigr)
 		\sinh\!\left(\frac{\ell(\gamma)}{2}\right),
 	\end{equation*}
	\begin{equation*}
 		\sinh\!\left(\frac{d(q, \gamma q)}{2}\right)
 		=
 		\cosh \bigl(d(q, c)\bigr)
 		\sinh\!\left(\frac{\ell(\gamma)}{2}\right).
 	\end{equation*}

 We notice that \[ d(z, c )  \le d(z, q )+ d(q,c) = t+1 - t_1 + d(q,c) \le 1 + d(q,c).\] We obtain 
 \[
 \sinh\!\left(\frac{d(z, \gamma z)}{2}\right) \leq e \sinh\!\left(\frac{d(q, \gamma q)}{2}\right).
 \]

 Since $d(q,\gamma q)\le 1$, we have
 \[
 d(z, \gamma z) \le 2 \sinh\!\left(\frac{d(z, \gamma z)}{2}\right) \le 2 e \sinh\!\left(\frac{d(q, \gamma q)}{2}\right) \le 2 e d(q,\gamma q)\le 6 d(q,\gamma q).
 \] 

 We get
 \[
 	\forall \, t_1 -1 < t < t_1, \quad d_1(g_{t}\tilde v, g_t \tilde u) \le 15 d(q,\gamma q).
 	\]
 	
\end{enumerate}

Finally, we observe that 
\begin{align*}
	d_1\!\left(g_{t+\tau} \tilde v ,\, g_t\tilde u\right) & \le d_1\!\left(g_{t+\tau} \tilde v ,\, g_t\tilde v\right) + d_1\!\left(g_{t} \tilde v ,\, g_t\tilde u\right)\\
	& =  2|\tau| + d_1\!\left(g_{t} \tilde v ,\, g_t\tilde u\right) \\
	&\le  2d(q,\gamma q) + d_1\!\left(g_{t} \tilde v ,\, g_t\tilde u\right),
\end{align*}
thanks to Proposition \ref{bound_wind}. Similarly we have
\[
d_1\!\left(g_{t+\tau} \tilde v ,\, \gamma g_t\tilde u\right) \le 2 d(q,\gamma q) + d_1\!\left(g_{t} \tilde v ,\, \gamma g_t\tilde u\right).
\]
Together with the previous bounds, these imply the statement.
 \end{proof}


\section{Proof of Theorem \ref{thm2}}\label{sec:proof}

\subsection{Rays intersecting arbitrarily short hypercycles}\label{sec:vanishing}

Let $S=\Gamma\backslash\H$ be a hyperbolic surface and let $u\in T^1S$ be such that the geodesic ray $u(\R^+)$ meets closed hypercycles of arbitrarily small length. Notice that each closed hypercycle has a closed geodesic in its free homotopy class, which is shorter. The case where $u(\R^+)$ meets infinitely many of these closed geodesics falls in the first case of Proposition \ref{trichotomy} and the conclusion of Theorem \ref{thm2} follows from \cite{BurniolDalboHerrero2026}. Therefore, we may assume that $u(\R^+)$ does not meet the closed geodesics associated to these hypercycles.

Recall that $\mathcal H(\gamma,q)$ denotes the hypercycle of $\gamma$ through $q$ (Section \ref{sec:hypercycles}). If $\gamma$ is in good position with respect to $\tilde u$ and $q$ is the point of $\tilde u(\reals)$ closest to the axis of $\gamma$, then $\tilde u(\reals)$ is tangent to $\mathcal H(\gamma,q)$ at $q$.

\begin{figure}[htbp]
\centering
\begin{tikzpicture}[scale=1.6,>=Stealth,line join=round]

  \draw[->] (-1.3,0) -- (5.8,0) node[below right] {$\partial\H=\mathbb{R}$};

  \draw[very thick,blue,->] (0,0) -- (0,5.25);
  \node[blue,left] at (0,4.9) {$\tilde u(\mathbb{R})$};
  \draw[very thick,blue,->] (0,3.0) -- (0,3.7);

  \draw[very thick,red] (1,0) arc (180:0:1.5);
  \node[red] at (2.5,0.72) {$c_n$};

  \draw[very thick,green!55!black] (1,0) arc (233.13:-53.13:2.5);
  \node[green!55!black,above] at (2.5,4.5) {$\mathcal H(\gamma_n,q_n)$};

  \draw[dashed,thick,gray] (0,2) arc (90:36.87:2);

  \fill (0,2) circle (1.5pt) node[left] {$q_n$};      
  \fill (1.6,1.2) circle (1.3pt);                     
  \fill (0,1) circle (1.2pt) node[left] {$i$};        
  \fill (1,0) circle (1.3pt) node[below] {$\gamma_n^-$};
  \fill (2,0) circle (1.1pt) node[below] {$x_n$};
  \fill (4,0) circle (1.3pt) node[below] {$\gamma_n^+$};

\end{tikzpicture}
\caption{The geodesic $\tilde u(\mathbb{R})$, the axis $c_n$ of $\gamma_n$, and the hypercycle $\mathcal H(\gamma_n,q_n)$ tangent to $\tilde u(\mathbb{R})$ at $q_n$.}
\label{fig:hypercycle}
\end{figure}
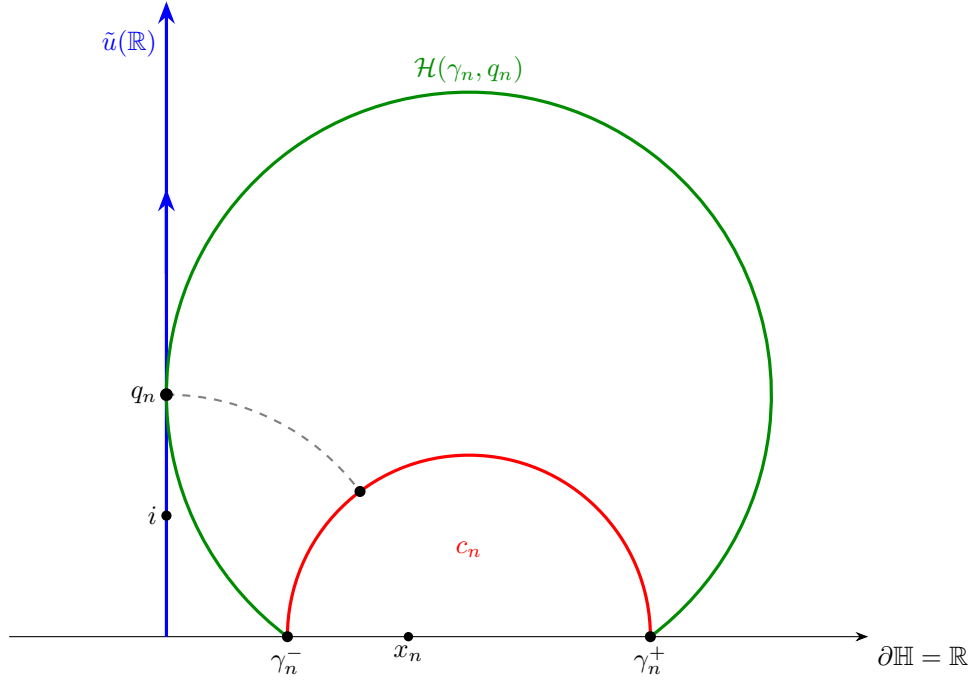

\begin{proposition}\label{normalization}
	Let $S=\Gamma\backslash\H$ be a hyperbolic surface and let $u\in T^1S$ be such that $u(\R^+)$ meets closed hypercycles of arbitrarily small length without meeting the associated closed geodesics. Then, up to conjugating $\Gamma$ by an isometry of $\H$, there exist a lift $\tilde u$ of $u$ with $\tilde u(0)=i$ and $\tilde u(+\infty)=\infty$, and a sequence $(\gamma_n)_{n\ge 0}$ of hyperbolic isometries in $\Gamma$ such that
	\begin{enumerate}
		\item $\tilde u$ is in good position with respect to each $\gamma_n$,
		\item the point $q_n=\tilde u(t_n)$ of $\tilde u(\reals)$ closest to the axis of $\gamma_n$ is such that $(t_n)_{n\ge 0}$ is increasing and tends to $+\infty$,
		\item $0<\gamma_n^-<\gamma_n^+<\gamma_{n+1}^-$ for all $n\ge 0$,
		\item the hypercycles $\mathcal H(\gamma_n,q_n)$ are pairwise disjoint,
		\item $d(q_n,\gamma_n q_n)$ tends to $0$.
	\end{enumerate}
\end{proposition}

\begin{proof}
By hypothesis, there are times $s_n\ge0$ such that $u(s_n)$ belong to a hypercycle whose length tend to $0$ when $n$ tends to $+\infty$.

Up to conjugating the group $\Gamma$ by an isometry of $\H$, we can assume that $u$ lifts to $\tilde u \in T^1 \H$ such that $\tilde u(0)= i$ and $\tilde u(+\infty)= \infty$. For every $n$, there exists a primitive hyperbolic isometry $\gamma_n$ such that the hypercycle $\mathcal{H}(\gamma_n, \tilde u(s_n)) $ of $\gamma_n$ through $\tilde u(s_n)$ project to the closed hypercycles whose length tend to $0$. In particular, $\varepsilon_n:=d(\tilde u(s_n),\gamma_n\tilde u(s_n))$ is smaller than the length of the hypercycle, so it also tends to $0$.

Since $\ell(\gamma_n)\le \varepsilon_n\to 0$, up to passing to a subsequence, the isometries $\gamma_n$ can be taken to be pairwise distinct. It can also be assumed that their fixed points are different from $0$ and $\infty$.

	Let $c_n$ be the axis of $\gamma_n$. By hypothesis $\tilde u (\R^+)$ does not intersect $c_n$. 
	Since the fixed points of $\gamma_n$ are different from $0$ and $\infty$, the geodesics $\tilde u(\reals)$ and $c_n$ are not asymptotic. Let $q_n=\tilde u(t_n)$ be the point of $\tilde u(\reals)$ closest to $c_n$. By \eqref{eq:displacement}, the displacement $d(z,\gamma_n z)$ is smaller the closer $z$ lies to $c_n$. In particular,
	\[
	d(q_n,\gamma_n q_n)\le d\bigl(\tilde u(s_n),\gamma_n\tilde u(s_n)\bigr)=\varepsilon_n\longrightarrow 0.
	\]

	We claim that $t_n>0$ for $n$ large enough. Otherwise, if $t_n\le 0$, then $t_n\le 0\le s_n$, so $\tilde u(0)=i$ is closer to $c_n$ than $\tilde u(s_n)$. By the property of the displacement function, we would get
	\[
	d\bigl(i ,\gamma_n i \bigr)\le d\bigl(\tilde u(s_n),\gamma_n\tilde u(s_n)\bigr)=\varepsilon_n,
	\]
	which if it happened for infinitely many $n$ would contradict the discreteness of $\Gamma$. Up to taking a subsequence, we assume $t_n>0$. Moreover, $q_n$ does not lie on $c_n$, thus $\tilde u(\reals)$ stays at positive distance from $c_n$.

	Observe that $\tilde u (\reals)$ is tangent to the hypercycle of $\gamma_n$ through $q_n$. Replacing $\gamma_n$ by $\gamma_n^{-1}$ if necessary, the orientation of $\tilde u(\reals)$ agrees with the one induced by $\gamma_n$ on the hypercycle through $q_n$. Together with $t_n>0$, this shows that $\tilde u$ is in good position with respect to $\gamma_n$.

	Moreover, $t_n\to+\infty$. Otherwise, up to passing to a subsequence, the points $q_n$ stay in a compact subset of $\H$, and $d(q_n,\gamma_n q_n)\to 0$ with pairwise distinct $\gamma_n$ would contradict the discreteness of $\Gamma$. Up to passing to a subsequence, we may assume that $(t_n)_{n\ge 0}$ is increasing.

	The axis $c_n$ does not meet $\tilde u(\reals)$, hence it is contained in one of the two half-planes $\{x>0\}$ or $\{x<0\}$. Up to taking a subsequence and conjugating again the group $\Gamma$ by the reflection $z\mapsto-\bar z$, which fixes $\tilde u(\reals)$ pointwise and preserves good position, we may assume that they all lie in $\{x>0\}$. Then $\gamma_n^->0$ and $\gamma_n^+>0$, and good position forces the attracting fixed point to be on the right of the repelling one, so that $0<\gamma_n^-<\gamma_n^+$.

	The geodesic realizing the distance between $\tilde u(\reals)$ and $c_n$ is orthogonal to both, and it passes through the point $q_n$. Hence it is a Euclidean half-circle centered at $0$. One of the end points $x_n$ of this geodesic is between $\gamma_n^-$ and $\gamma_n^+$, so $\gamma_n^-< x_n < \gamma_n^+$. By definition, $q_n=\tilde u(t_n)=i e^{t_n}$, then $x_n= e^{t_n}$ tends to $+\infty$, and so does $\gamma_n^+$. We claim that $\gamma_n^-$ also tends to $+\infty$. Otherwise, up to passing to a subsequence, we may assume that $\gamma_n^-$ converges to some $\xi \in \reals $. Then the axes $c_n$ converge to the geodesic with endpoints $\xi$ and $+\infty$. Picking a point $z$ on this geodesic, it is accumulated by points $z_n $ on $c_n$. Since $\ell (\gamma_n)$ tends to $0$, we have $d(z, \gamma_n z)\to 0$, which contradicts the fact that $\Gamma$ is discrete. Hence, we must have $\lim_n \gamma_n^- = +\infty$.

	\medskip
	The hypercycle $\mathcal H(\gamma_n,q_n)$ is an arc of the Euclidean circle through $\gamma_n^-$, $\gamma_n^+$ and $q_n$, which is tangent to the imaginary axis at $q_n$ (see Figure \ref{fig:hypercycle}).
	 Then the Euclidean distance between $0$ and $\mathcal H(\gamma_n,q_n)$ tends to $+\infty$. For this we can observe that this distance is greater than the distance from $0$ to the Euclidean circle of radius $\gamma_n^-$ centered at $\gamma_n^-+ \gamma_n^- i $.
	This allows us to extract a subsequence such that the hypercycles $\mathcal H(\gamma_n,q_n)$ are pairwise disjoint and $\gamma_n^+<\gamma_{n+1}^-$ for all $n\ge 0$.
\end{proof}

From now on we fix the vector $u\in T^1S$, its lift $\tilde u$ with $\tilde u(0)=i$ and $\tilde u(+\infty)=\infty$, and the sequence $(\gamma_n)_{n\ge 0}$ of hyperbolic isometries given by Proposition \ref{normalization}.

\subsection{Construction of the vector \texorpdfstring{$w_\alpha$}{w-alpha} in \texorpdfstring{$\overline{H^{ss} u}$}{the closure of Hss u}}\label{subsec_construction}

Let $(\alpha_n)_{n\ge 0}$ be a subsequence of $(\gamma_n)_{n\ge 0}$. We write $\beta_n=\alpha_0\cdots\alpha_n\in\Gamma$, and we denote by $\alpha_n^-,\alpha_n^+$ the repelling and attracting fixed points of $\alpha_n$, so that $0<\alpha_n^-<\alpha_n^+<\alpha_{n+1}^-$. For each $n\ge 0$, let $q_n^\alpha$ be the point of $\tilde u(\reals)$ closest to the axis of $\alpha_n$, and set
\[
\delta_n^\alpha=d(q_n^\alpha,\alpha_n q_n^\alpha).
\]
Since $(\alpha_n)$ is a subsequence of $(\gamma_n)$, we have $\delta_n^\alpha\to 0$. We introduce the vectors $\tilde v_n\in T^1\H$ defined by
\[
\tilde v_n(0)=i,\qquad \tilde v_n(+\infty)=\beta_n(\infty).
\]

\begin{lemma}\label{lemma_vectors}
	The sequence $(\tilde v_n)_{n\ge 0}$ converges to a vector $\tilde v_\alpha\in T^1\H$ defined by $\tilde v_\alpha(0)=i$ and $\tilde v_\alpha(+\infty)=\lim_{n\to +\infty}\beta_n(\infty)$.
\end{lemma}

\begin{proof}
	Since every $\tilde v_n$ is based at $i$, it suffices to prove that the sequence $(\beta_n(\infty))_{n\ge 0}\subset\partial\H$ converges. Each $\alpha_i$ is a hyperbolic isometry with repelling and attracting fixed points $0<\alpha_i^-<\alpha_i^+$. Since $\infty$ lies outside $[\alpha_i^-,\alpha_i^+]$, it is attracted by $\alpha_i$ toward $\alpha_i^+$, so that $\alpha_i^+<\alpha_i\infty<+\infty$. Moreover, $\alpha_i$ preserves the order on the interval $(\alpha_i^+,+\infty)$, hence $\alpha_i^+<x<y$ implies $\alpha_i^+<\alpha_i x<\alpha_i y$. Using the nesting $\alpha_i^+<\alpha_{i+1}^-$, we obtain for every $n\ge 1$
	\[
	0<\alpha_{n-1}\alpha_n\infty<\alpha_{n-1}\infty .
	\]
	Applying $\alpha_0,\dots,\alpha_{n-2}$ successively, which preserve the order on the relevant intervals, yields
	\[
	0<\beta_n(\infty)<\beta_{n-1}(\infty).
	\]
	Hence $(\beta_n(\infty))_{n\ge 0}$ is a decreasing sequence of positive numbers, and therefore converges.
\end{proof}

The ray $\beta_n^{-1}\tilde v_n(\reals^+)$ issues from $\beta_n^{-1}i$ and points to $\beta_n^{-1}(\beta_n\infty)=\infty$; it is therefore the vertical geodesic ray starting at $\beta_n^{-1}i$. The next lemma places this ray in good position with respect to $\alpha_{n+1}$ and compares its displacement with $\delta_{n+1}^\alpha$. This is the geometric point where the present case departs from \cite{BurniolDalboHerrero2026}: the ray does not cross the axis of $\alpha_{n+1}$ but remains tangent to one of its hypercycles.

\begin{lemma}\label{good_position_vn}
	For every $n\ge 0$, the vector $\beta_n^{-1}\tilde v_n$ is in good position with respect to $\alpha_{n+1}$. Moreover, if $q_{n+1}'$ denotes the point of $\beta_n^{-1}\tilde v_n(\reals)$ closest to the axis of $\alpha_{n+1}$, then
	\[
	d(q_{n+1}',\alpha_{n+1}q_{n+1}')\le \delta_{n+1}^\alpha .
	\]
\end{lemma}

\begin{proof}
	For $n\ge 0$, let $R_n$ be the closed region of the first quadrant bounded by the segment of the imaginary axis between $0$ and $q_n^\alpha$, the arc of the hypercycle $\mathcal H(\alpha_n,q_n^\alpha)$ between $q_n^\alpha$ and $\alpha_n^-$, and the segment $[0,\alpha_n^-]$ of $\partial\H$. Since the hypercycles $\mathcal H(\alpha_n,q_n^\alpha)$ are pairwise disjoint and $q_n^\alpha=\tilde u(t_n^\alpha)$ with $(t_n^\alpha)$ increasing, we have $R_n\subset R_{n+1}$ for all $n\ge 0$. Moreover, since $\mathcal H(\alpha_n,q_n^\alpha)$ is invariant under $\alpha_n$ and $\alpha_n^{-1}$ moves the points of $\partial\H$ toward its attracting fixed point $\alpha_n^-$, one checks that $\alpha_n^{-1}(R_n)\subset R_n$.

	We prove by induction that $\beta_n^{-1}i\in R_n$ for all $n\ge 0$. For $n=0$ we have $\alpha_0^{-1}(i)\in\alpha_0^{-1}(R_0)\subset R_0$, because $i\in R_0$. Assume $\beta_n^{-1}i\in R_n$. Then
	\[
	\beta_{n+1}^{-1}i=\alpha_{n+1}^{-1}\beta_n^{-1}i\in\alpha_{n+1}^{-1}(R_n)\subset\alpha_{n+1}^{-1}(R_{n+1})\subset R_{n+1}.
	\]

	Since $\beta_n^{-1}i\in R_n\subset R_{n+1}$, the vertical ray $\beta_n^{-1}\tilde v_n(\reals^+)$ lies on the left of $\alpha_n^-$, hence it stays at positive distance from the axis of $\alpha_{n+1}$, and is tangent to a hypercycle of $\alpha_{n+1}$ enclosed by $\mathcal H(\alpha_{n+1},q_{n+1}^\alpha)$. 
	The orientation of $\beta_n^{-1}\tilde v_n(\R^+)$ agrees with the one induced by $\alpha_{n+1}$, so $\beta_n^{-1}\tilde v_n$ is in good position with respect to $\alpha_{n+1}$. Finally, the point $q_{n+1}'$ realizing the distance to the axis lies on a hypercycle enclosed by $\mathcal H(\alpha_{n+1},q_{n+1}^\alpha)$, hence closer to the axis than $q_{n+1}^\alpha$. As the displacement function decreases with the distance to the axis, we conclude
	\[
	d(q_{n+1}',\alpha_{n+1}q_{n+1}')\le d(q_{n+1}^\alpha,\alpha_{n+1}q_{n+1}^\alpha)=\delta_{n+1}^\alpha. \qedhere
	\]
\end{proof}

By Lemma \ref{good_position_vn}, the winding $\mathrm{Wind}_{\alpha_{n+1}}(\beta_n^{-1}\tilde v_n)$ is defined, and it is the vector based at $\beta_n^{-1}i$ pointing at $\alpha_{n+1}\infty$. On the other hand, $\beta_n^{-1}\tilde v_{n+1}$ is based at $\beta_n^{-1}i$ and points at $\beta_n^{-1}(\beta_{n+1}\infty)=\alpha_{n+1}\infty$. Hence, for all $n\ge 0$,
\[
\beta_n^{-1}\tilde v_{n+1}=\mathrm{Wind}_{\alpha_{n+1}}(\beta_n^{-1}\tilde v_n).
\]
By the definition of the winding time (Definition \ref{def_winding_time}) applied to $\alpha_{n+1}$ and $\beta_n^{-1}\tilde v_n$, together with Proposition \ref{bound_wind} and Lemma \ref{good_position_vn},
\begin{equation}\label{winding_time_seq}
	\bigl|B_\infty(\alpha_{n+1}^{-1}\beta_n^{-1}i,\,\beta_n^{-1}i)\bigr|
	=\bigl|\tau_{\alpha_{n+1},\,\beta_n^{-1}\tilde v_n}\bigr|
	\le d(q_{n+1}',\alpha_{n+1}q_{n+1}')
	\le \delta_{n+1}^\alpha .
\end{equation}

We define
\[
r_n=B_\infty(\beta_n^{-1}i,\,i).
\]
By construction $g_{r_n}\beta_n^{-1}\tilde v_n\in H^{ss}\tilde u$, so its projection satisfies $g_{r_n}v_n\in H^{ss}u$.

\begin{proposition}\label{convergence_winding_time}
	Assume that $\sum_{n\ge 0}\delta_n^\alpha<+\infty$. Then the sequence $(r_n)_{n\ge 0}$ converges to a real number $r_\alpha$, and $|r_n|\le\sum_{i\ge 0}\delta_i^\alpha$ for all $n\ge 0$; hence $|r_\alpha|\le\sum_{i\ge 0}\delta_i^\alpha$.
\end{proposition}

\begin{proof}
	By definition of $r_n$ and \eqref{winding_time_seq},
	\[
	|r_{n+1}-r_n|=\bigl|B_\infty(\alpha_{n+1}^{-1}\beta_n^{-1}i,\,\beta_n^{-1}i)\bigr|
	=\bigl|\tau_{\alpha_{n+1},\,\beta_n^{-1}\tilde v_n}\bigr|\le \delta_{n+1}^\alpha .
	\]
	Hence $|r_m-r_n|\le\sum_{i=n+1}^{m}\delta_i^\alpha$, so $(r_n)$ is a Cauchy sequence and converges. Moreover
	\[
	|r_n|\le |r_0|+\sum_{i=0}^{n-1}|r_{i+1}-r_i|\le\sum_{i=0}^{n-1}\delta_i^\alpha\le\sum_{i\ge 0}\delta_i^\alpha,
	\]
	using $|r_0|=|\tau_{\alpha_0,\tilde u}|\le\delta_0^\alpha$ (Proposition \ref{bound_wind}). Letting $n\to+\infty$ gives the last inequality.
\end{proof}

\subsection{Construction of the vector \texorpdfstring{$w_\alpha$}{w-alpha} in \texorpdfstring{$W^{ss} u$}{Wss u}}\label{subsec_Wss}

Throughout this subsection we fix a subsequence $(\alpha_n)_{n\ge 0}$ of $(\gamma_n)_{n\ge 0}$ such that $\delta_n^\alpha\le 1$ for all $n$ and $\sum_{n\ge 0}\delta_n^\alpha<+\infty$, and we write $\Delta=\sum_{n\ge 0}\delta_n^\alpha$.

By Proposition \ref{convergence_winding_time} the winding times $r_n$ converge to $r_\alpha$, and by Lemma \ref{lemma_vectors} the vectors $\tilde v_n$ converge to $\tilde v_\alpha$. We set, for $n\ge 0$,
\[
\tilde w_n=g_{r_n}\tilde v_n\in T^1\H ,\qquad
\tilde w_\alpha=g_{r_\alpha}\tilde v_\alpha=\lim_{n\to+\infty}\tilde w_n\in T^1\H .
\]
We denote by $v_n$, $v_\alpha$, $w_n$, $w_\alpha$ the projections to $T^1S$ of $\tilde v_n$, $\tilde v_\alpha$, $\tilde w_n$, $\tilde w_\alpha$, respectively. We remark that,since $\beta_n^{-1}\tilde w_n=g_{r_n}\beta_n^{-1}\tilde v_n\in H^{ss}\tilde u$, the vectors $ w_n$ lie in $ H^{ss} u$ and $w_n $ converges to $w_\alpha$.

Our goal is to show that in fact $w_\alpha\in W^{ss}u$. For this we show that the distance between $g_t w_n$ and $g_t u$ tends to zero uniformly in $n$ and then we take the limit in $n$.

\begin{proposition}\label{uniform_tube}
	For every $\varepsilon>0$ there exists a time $T>0$ such that
	\[
	d_1\bigl(g_t u,\,g_t w_n\bigr)<\varepsilon\qquad\text{for all }t>T\text{ and all }n\ge 0 .
	\]
\end{proposition}

\begin{proof}
	Fix $\varepsilon>0$. Since $\sum_k\delta_k^\alpha<+\infty$, we may fix $N$ with $17\sum_{k>N}\delta_k^\alpha<\varepsilon/2$. We will use that, if $n>N$, then 
	\begin{equation} \label{estimate}
	d_1(g_t u, g_t w_n) \le d_1(g_t u, g_t w_N) + \sum_{k=N}^{n-1} d_1(g_t w_{k+1}, g_t w_k).
	\end{equation}
	
	Recall that $\beta_k^{-1}\tilde v_{k+1}=\mathrm{Wind}_{\alpha_{k+1}}(\beta_k^{-1}\tilde v_k)$ and $\tau_{\alpha_{k+1},\beta_k^{-1}\tilde v_k}=r_{k+1}-r_k$, and that, by Lemma \ref{good_position_vn}, the point of $\beta_k^{-1}\tilde v_k(\reals)$ closest to the axis of $\alpha_{k+1}$ has displacement at most $\delta_{k+1}^\alpha\le 1$. Applying the Key Proposition \ref{key_prop} to $\beta_k^{-1}\tilde v_k$ and $\alpha_{k+1}$ and then projecting to $T^1S$, we obtain 
	\[
	d_1\bigl(g_{s+r_{k+1}-r_k}v_{k+1},\,g_s v_k\bigr)\le 17\,\delta_{k+1}^\alpha, \qquad \forall s\ge 0.
	\]

	Substituting $s=t+r_k$ gives 
	\[
	d_1(g_{t}w_{k+1}, g_{t}w_k) = d_1(g_{t+r_{k+1}}v_{k+1}, g_{t+r_k}v_k) \le 17 \delta_{k+1}^\alpha
	\] for all $t \ge -r_k$. Since $|r_k| \le \Delta$, this holds for all $t \ge \Delta$ in particular. Therefore the second term in \eqref{estimate} is bounded by $17\sum_{k=N}^{n-1}\delta_{k+1}^\alpha < \varepsilon /2 $ for all $t \ge \Delta$.

	To bound the first term in \eqref{estimate} or to treat the case $n \le N$, we may observe again that $w_i \in H^{ss} u$, so $d_1(g_t w_i, g_t u) \to 0$ as $t \to +\infty$. Hence there exists $T \ge  \Delta$ such that $d_1(g_t w_i, g_t u) < \varepsilon / 2$ for all $t > T$. This concludes the proof.
\end{proof}

Since $w_\alpha=\lim_{n\to+\infty}w_n$, letting $n\to+\infty$ in Proposition \ref{uniform_tube} gives, for every $\varepsilon>0$ and every $t>T$,
\[
d_1(g_t w_\alpha,g_t u)=\lim_{n\to+\infty}d_1(g_t w_n,g_t u)\le\varepsilon .
\]
Hence $\lim_{t\to+\infty}d_1(g_t w_\alpha,g_t u)=0$, that is, $w_\alpha\in W^{ss}u$.

\subsection{Construction of the vector \texorpdfstring{$w_\alpha$}{w-alpha} in \texorpdfstring{$W^{ss} u\setminus H^{ss}u$}{Wss u minus Hss u}}\label{subsec_uncountable}

We denote by $\Sigma$ the set of subsequences $\alpha=(\alpha_n)_{n\ge 0}$ of $(\gamma_n)_{n\ge 0}$ such that
\[
\delta_n^\alpha\le 1\ \ \text{for all }n\ge 0\qquad\text{and}\qquad\sum_{n\ge 0}\delta_n^\alpha<+\infty .
\]
By Proposition \ref{uniform_tube}, every $\alpha\in\Sigma$ satisfies $w_\alpha\in W^{ss}u$.
Since $\delta_n=d(q_n,\gamma_n q_n)\to 0$, the set $\Sigma$ is nonempty. For $\alpha\in\Sigma$ we set
\[
\xi(\alpha)=\tilde w_\alpha(+\infty)=\lim_{n\to+\infty}\beta_n(\infty)\in\reals,
\]
which is a positive number by Lemma \ref{lemma_vectors}.

\begin{proposition}\label{uncountable}
	The image $\xi(\Sigma)$ is uncountable.
\end{proposition}

\begin{proof}
	We use the following property, which follows from the dynamics of the $\alpha_n$ as in the proof of Lemma \ref{lemma_vectors}: for every $\alpha\in\Sigma$ and every $n\ge 0$,
	\begin{equation}\label{above_fixed_point}
		\alpha_n\alpha_{n+1}\cdots\alpha_m(\infty)>\alpha_n^+\quad\text{for all }m\ge n,
		\qquad\text{hence}\qquad
		\alpha_{n-1}^{-1}\cdots\alpha_0^{-1}\bigl(\xi(\alpha)\bigr)\ge\alpha_n^+ .
	\end{equation}

	Assume, for contradiction, that $\xi(\Sigma)$ is countable and enumerate it as $\{x_0,x_1,\dots\}$. We construct $\alpha'=(\alpha_n')\in\Sigma$ with $\xi(\alpha')\neq x_i$ for all $i$.

	Choose $\alpha_0'=\gamma_{k_0}$ with $\delta_{0}^{\alpha'}<1$ and $\alpha_0'^+>x_0$. Inductively, having chosen $\alpha_0',\dots,\alpha_n'$, pick $\alpha_{n+1}'=\gamma_{k_{n+1}}$ with $k_{n+1}>k_n$ such that
	\[
	\delta_{n+1}^{\alpha'}<\frac{1}{2^{n+1}}
	\qquad\text{and}\qquad
	\alpha_{n+1}'^+>\alpha_n'^{-1}\cdots\alpha_0'^{-1}(x_{n+1}).
	\]
	This is possible since $\gamma_k^+\to+\infty$ and $\delta_k\to 0$. Since $\sum_{n\ge 0}\delta_n^{\alpha'}<\sum_{n\ge 0}2^{-n}<+\infty$ and $\delta_n^{\alpha'}\le 1$, the resulting sequence $\alpha'$ belongs to $\Sigma$.

	By \eqref{above_fixed_point} and the choice of $\alpha_n'^+$, we have for $n\ge 1$
	\[
	\alpha_{n-1}'^{-1}\cdots\alpha_0'^{-1}\bigl(\xi(\alpha')\bigr)\ge\alpha_n'^+>\alpha_{n-1}'^{-1}\cdots\alpha_0'^{-1}(x_n),
	\]
	whence $\xi(\alpha')\neq x_n$. Finally $\xi(\alpha')>\alpha_0'^+>x_0$, so $\xi(\alpha')\neq x_0$. This contradicts $\xi(\alpha')\in\xi(\Sigma)=\{x_i\}$.
\end{proof}

\begin{corollary}\label{cor2}
	There exists a subsequence $\alpha\in\Sigma$ such that $w_\alpha\notin H^{ws}u$. Moreover, $W^{ss}u$ contains an uncountable union of pairwise distinct stable horocycles.
\end{corollary}

\begin{proof}
	If $w\in T^1S$ belongs to $H^{ws}u$, then $\tilde w(+\infty)\in\Gamma\tilde u(+\infty)=\Gamma\infty$ for any lifts. Since $\Gamma$ is countable and $\xi(\Sigma)$ is uncountable by Proposition \ref{uncountable}, the set $\xi(\Sigma)$ meets uncountably many $\Gamma$-orbits in $\partial\H$. Hence there exists $\alpha\in\Sigma$ with $\xi(\alpha)=\tilde w_\alpha(+\infty)\notin\Gamma\infty$, that is, $w_\alpha\notin H^{ws}u$. By Proposition \ref{uniform_tube}, all these $w_\alpha$ lie in $W^{ss}u$, and their forward endpoints lie in distinct $\Gamma$-orbits, so the stable horocycles $H^{ss}w_\alpha\subset W^{ss}u$ are pairwise distinct and uncountable in number.
\end{proof}

\section*{Declaration on the use of AI}
During the preparation of this work the authors used Claude (Anthropic) to assist with the writing and the proofreading of the manuscript. The mathematical content and the arguments are entirely the authors' own. The authors have reviewed all the text and take full responsibility for the content of the article.

\bibliographystyle{alpha}
\bibliography{general_library}

\end{document}